\documentclass[12pt, english]{amsart}
\usepackage[T1]{fontenc}
\usepackage[latin9]{inputenc}
\usepackage{geometry}
\usepackage{amsthm}
\usepackage{esint}
\usepackage{color}
\usepackage{ulem}
\makeatletter
\numberwithin{equation}{section}
\numberwithin{figure}{section}

\makeatother

\usepackage{babel}

\newcommand{\be}{\begin{equation}}
\newcommand{\ee}{\end{equation}}

\newtheorem{thm}{Theorem}
\newtheorem{defn}[thm]{Definition}
\newtheorem{rmk}[thm]{Remark}
\newtheorem{prop}[thm]{Proposition}
\newtheorem{lem}[thm]{Lemma}
\newtheorem{thm*}{Theorem}

\newtheorem{quest}[thm]{Question}

\newcommand{\Ham}{\operatorname{Ham}_3}
\newcommand{\RotSym}{\operatorname{RotSym}_3}

\begin{document}

\title{Bartnik's Mass and Hamilton's Modified Ricci Flow 
}

\author{Chen-Yun Lin}
\address{Chen-Yun Lin\\
Department of Mathematics\\
Lehman College, CUNY}
\email{chenyun.lin@lehman.cuny.edu}

\author{Christina Sormani}
\address{Christina Sormani\\
Department of Mathematics\\
CUNY GC and Lehman College
}
\email{sormanic@gmail.com}
\thanks{Sormani and Lin were funded by $\textrm{NSF-DMS}$ 1309360.}


\maketitle

\begin{center}{\it Dedicated to Richard Hamilton on the occasion of his 70th birthday.}
\end{center}

\begin{abstract}
We provide estimates on the Bartnik mass of constant mean
curvature surfaces which are diffeomorphic to spheres and have positive
mean curvature. We prove that the Bartnik mass is bounded from above
by the Hawking mass and a new notion we call the asphericity mass.
The asphericity mass is defined by applying Hamilton's modified Ricci
flow and depends only upon the restricted metric of the surface and not
on its mean curvature. The theorem is proven by studying a class of
asymptotically flat Riemannian manifolds foliated by surfaces satisfying
Hamilton's modified Ricci flow with prescribed scalar curvature. Such
manifolds were first constructed by the first author in her dissertation
conducted under the supervision of M. T. Wang. We make a further
study of this class of manifolds which we denote Ham3, bounding the
ADM masses of such manifolds and analyzing the rigid case when the
Hawking mass of the inner surface of the manifold agrees with its ADM
mass.

{\color{blue}
After this paper was published, Hyun-Chul Jang observed that we dropped
a term in our calculations.  Tracking the
consequences throughout, we see that we need only
slightly change the
definition of the asphericity mass and then all statements of
our theorems, propositions, and lemmas 
remain the same with slight revisions to the proofs.   Pengzi Miao observed
that we need an assumption that $\Sigma$ has nonnegative Gauss curvature in Theorem 1.
We include all these corrections below as well as some clarifications
regarding the Gauss curvature in blue where they
are needed.   Hyun-Chul Jang and Pengzi Miao
have approved of our corrections and we have sent an erratum to the journal.
 }

\end{abstract}

\newpage
\section{Introduction}
\label{intro}

Two of the most important quasilocal masses studied in Riemannian General
Relativity are the Hawking mass and Bartnik mass of a surface, $\Sigma$, which is
diffeomorphic to a sphere, has positive mean curvature, and lies in an asymptotically
flat three-dimensional Riemannian manifold, $M$. The manifold, $M$,
has nonnegative scalar curvature and no closed interior minimal surfaces. It
may have a boundary, as long as the boundary is a minimal surface and is
outward minimizing. We will use $\mathcal{PM}$ to denote the class of such manifolds,
$M$.

In this paper, we relate these two quasilocal masses with a third quantity
that we call the ``asphericity mass''. We prove this new quantity depends only
on the intrinsic geometry of $\Sigma$ and is $0$ if and only if $\Sigma$ is a standard sphere;
thus, it is a measure of ``asphericity''. We consider it to be a ``mass'' because it
scales like mass and is related to a difference between two quasilocal masses.
However, it is not a quasilocal mass.

Before describing our results, we give a very brief review of the key definitions
needed to state our theorems. We apologize that we cannot completely
survey the results of the many mathematicians and physicists that have contributed
to research on mass in general relativity. We review only the results
related to our class of three dimensional manifolds $\mathcal{PM}$ that are directly related
to the work in this paper. We do not state the full generality of all
theorems proven in the papers we review nor related papers that extend these
results.

In 1961, Arnowitt-Deser-Misner introduced the ADM mass, which we
denote by $\mathrm{m}_{\mathrm{ADM}}(M)$, for asymptotically flat three-dimensional manifolds, including
$M \in \mathcal{PM}$ \cite{ADM61}. Note that the Riemannian Schwarzschild manifold,
$M_{Sch,m}$, for a black hole in a vacuum of mass, $m$, with metric
\be
g= (1-\frac{2m}{r})^{-1}dr^2 + r^2 g_{\mathbb{S}^2}
\ee
has scalar curvature $=0$ and $\mathrm{m}_{\mathrm{ADM}} (M) =m$. In 1968, 
 Hawking \cite{Haw68} introduced the Hawking mass
\be
\mathrm{m}_{\mathrm{H}}\left(\Sigma\right)=\sqrt{\frac{area\left(\Sigma\right)}{16\pi}}\left(1-\frac{1}{16\pi}\oint_{\Sigma}H^{2}d\sigma \right),
\ee
which approaches the ADM mass for large coordinate spheres, 
$\Sigma_r$:
\be
\mathrm{m}_{\mathrm{ADM}} (M)= \lim_{r\rightarrow \infty} \mathrm{m}_{\mathrm{H}}(\Sigma_r).
\ee
Note that on the Riemannian Schwarzschild manifold, $M_{Sch,m}$, the Hawking mass of all rotationally
symmetric spheres is $\mathrm{m}_{\mathrm{H}}(\Sigma)=m \ge 0$. More generally, 
when $M \in \mathcal{PM}$ is rotationally symmetric,
\be
g= (u(r))^2dr^2 + r^2 g_{\mathbb{S}^2},
\ee
the Hawking mass of level sets of $r$, $\Sigma_r$,
is nonnegative and increases to $\mathrm{m}_{\mathrm{ADM}}(M)$. 
Even without rotational symmetry, Geroch proved that for $M \in 
\mathcal{PM}$ and $\Sigma_t \subset M$ evolving by smooth inverse mean curvature flow, the Hawking
mass increases (see the appendix to \cite{Geroch73}).

Schoen-Yau proved in \cite{SY79} that for any
$M\in \mathcal{PM}$, one has $\mathrm{m}_{\mathrm{ADM}}(M)\ge 0$.  Huisken-Ilmanen proved the Penrose Inequality that
$\mathrm{m}_{\mathrm{ADM}}(M)\geq \mathrm{m}_{\mathrm{H}}(\partial M) \geq 0$ for $M\in \mathcal{PM}$ \cite{HI01}.   The Hawking mass itself is
not necessarily nonnegative, although it is clearly nonnegative for minimal surfaces.
Christodoulou and Yau \cite{CY86} proved that the Hawking mass is nonnegative for a stable 2-sphere with constant mean curvature.
However, Huisken-Ilmanen have an example of a $\Sigma \subset M$ where
$M\in \mathcal{PM}$ that has $\mathrm{m}_{\mathrm{H}}(\Sigma)<0$ \cite{HI02}.

The Bartnik mass was introduced in \cite{Bar89}. To define it, we first let $(\Omega^3,g)$ be the region enclosed by $\Sigma$.
For any bounded open connected region $(\Omega, g)$ with nonnegative scalar curvature, let $\mathcal{PM}(\Omega)$ be the set of ``admissible extensions'',
 $(M,g)\in \mathcal{PM}$ such that $\Omega$ embeds isometrically into $M$. Then the Bartnik's
definition for his mass is defined to be
 \be
\mathrm{m}_{\mathrm{B}}\left(\Omega\right)=\inf  \left\{\mathrm{m}_{\mathrm{ADM}}(M,g) : (M,g)\in \mathcal{PM}(\Omega) \right\},
\ee

Observe that by the Positive Mass Theorem, we have $\mathrm{m}_{\mathrm{B}} (\Omega) \geq 0$. Using
the inverse mean curvature flow, Huisken and Ilmanen \cite{HI01}  proved that if $\mathrm{m}_{\mathrm{B}} = 0 $
then $M$ is isometric to Euclidean space. Recall that Schoen-Yau proved that for any $M \in \mathcal{PM}$,
if $\mathrm{m}_{\mathrm{ADM}} (M) = 0$ then $M$ is isometric to Euclidean space. 
Recall that Schoen-Yau proved that for any $M \in \mathcal{PM}$, if $\mathrm{m}_{\mathrm{ADM}}(M)= 0$ then $M$ is isometric to Euclidean space.

As a quasilocal mass, the Bartnik mass may only depend on $\Sigma$ and how
$\Sigma$ embeds into $M^3$ but not on the interior region $\Omega$. Thus it is now standard
to define the Bartnik mass as follows:
 \be
\mathrm{m}_{\mathrm{B}}\left(\Omega\right)=\inf  \left\{\mathrm{m}_{\mathrm{ADM}}(M,g) : (M,g)\in \mathcal{PM'}(\Sigma) \right\},
\ee

Here $\mathcal{PM'}(\Sigma)$ is the set of ``admissible extensions'', $(M,g) \in \mathcal{PM'}$ such that
$ g|_{\partial (M\setminus \Omega)}=g|_{\partial\Omega}$ and $H_{\partial (M\setminus \Omega)} =H_{\partial\Omega}$. 
Here $\mathcal{PM'}$  are Lipschitz manifolds,
smooth away from Σ, that satisfy the same conditions as manifolds in $\mathcal{PM}$ where nonnegative scalar curvature is defined in the distributional sense across
$\Sigma$.

Note that the Bartnik mass is nonnegative, $\mathrm{m}_{\mathrm{B}}(\Omega) \geq 0$ This is an immediate
consequence of the fact that the Positive Mass Theorem holds for $M \in \mathcal{PM'}$ as proven by Shi and Tam  (See \cite[Theorem 3.1]{ShiTam02}) 
 and Miao \cite[Theorem 1]{Miao02}).

 Suppose that $\Sigma$ is isometric to a rescaled standard sphere and has constant
mean curvature. Then it is well known that $\mathrm{m}_{\mathrm{B}}(\Sigma)\le \mathrm{m}_{\mathrm{H}}(\Sigma)$ . To prove
this, one shows that any such $\Sigma$ includes a Riemannian Schwarzschild manifold
among its admissible extensions, and so $\mathrm{m}_{\mathrm{B}}(\Sigma) \leq  \mathrm{m}_{\mathrm{ADM}}(M_{Sch,m}) = \mathrm{m}_{\mathrm{H}}(\Sigma)$.
For completeness of exposition, we include the proof in our Appendix.
 
 In this paper, we consider constant positive mean curvature surfaces
which are not isometric to rescaled standard spheres. We construct an admissible
extension using Hamilton's modified Ricci flow \cite{Ham88}  as in the first author's
doctoral dissertation completed under the supervision of Mu-Tao Wang \cite{Lin13}.
We prove the following theorem:

\begin{thm}\label{thm-m-B}
Let $\Sigma$ be the boundary of a closed 3-dimensional region with nonnegative scalar curvature.
If $\Sigma$ \textcolor{blue}{has nonnegative Gauss curvature and } is a CMC surface diffeomorphic to a sphere, which has area $4\pi$ and positive mean curvature, $H$,
then
\be
\mathrm{m}_{\mathrm{B}}(\Sigma) \le m_{a\mathbb{S}}(\Sigma)+ \mathrm{m}_{\mathrm{H}}(\Sigma) 
\ee
where $m_{a\mathbb{S}}(\Sigma)$ is a nonnegative
constant defined using Hamilton's modified Ricci flow that we call the
asphericity mass.  It depends only upon the restriction of the
metric $g$ to the surface $\Sigma$.
If
\be
m_{a\mathbb{S}}(\Sigma)=0
\ee
then 
$\Sigma$ is isometric to a standard sphere, $(\mathbb{S}^2, g_{\mathbb{S}^2})$.
\end{thm}

Before we state the definition of the new asphericity mass
and our other theorems, we review Hamilton's modified 
Ricci flow.   Recall that for $(\Sigma, g_1)$ of dimension two,
Hamilton \cite{Ham88} defined the modified Ricci flow 
$(\Sigma_t,g_t)$ satisfying
\be \label{eq:RF}
     \left\{
             \begin{array}{cc}
                 \frac{\partial}{\partial t}g_{ij}=\left(r-2K\right)g_{ij}+2D_{i}D_{j}f=2M_{ij}\\
                     g\left(1,\cdot\right)=g_1\left(\cdot\right),
             \end{array}
     \right.
\ee
 where $K=K_t(x)$ is the Gauss curvature of $g_t$ at $x\in \Sigma_t$, and 
 \be \label{eq:r}
r=r_t=\frac{1}{Area(\Sigma_t)} \int_{\Sigma_t} 2K_t(x) d\mu=2
\ee
is the
mean scalar curvature, and $f=f(t,x)$ is the Ricci potential satisfying the
equation 
\be \label{eq:f}
\Delta f=2K-r
\ee
with mean value zero.
Thus the 2-tensor
\be\label{eq:M}
M_{ij}=(1-K)g_{ij} + D_iD_jf
\ee
is the trace-free part of $\mathrm{Hess}\left( f\right)$. 

Building upon this work of Hamilton,
Chow proved in \cite{Cho91}
that when $\Sigma$ is diffeomorphic to
a two dimensional sphere, then the modified Ricci flow 
exists for all time and $(\Sigma_t, g_t)$ converges to a standard sphere exponentially fast.   In fact $M$ converges to $0$ exponentially fast.

\begin{defn}\label{defn-aS}
The {\bf asphericity mass} of a surface $\Sigma$ of area $4\pi$ and diffeomorphic to a sphere is defined by
\be
m_{a\mathbb{S}}(\Sigma)= \lim_{t\to \infty} m_{a\mathbb{S}}(\Sigma,t),
\ee
where
\be
m_{a\mathbb{S}}(\Sigma,t)
=\frac{1}{2}  \int_{1}^{t}1-K_*(\tau) E(\tau,t) {\color{blue} + \frac {\tau |M|^{*2}}{2} E(\tau,t) }\,d\tau,
\ee
where
\be\label{E-tau}
E(\tau,t)=exp\left({-\int_{\tau}^{t}\frac{s\left|M\right|^{*2}(s)}{2}ds}\right) .
\ee

Here we have the infimum of the Gauss curvature
\be
K_{*}(\tau)=\inf\{K_\tau(x): \, x\in \Sigma_\tau\}
\ee
and the supremum of the norm of the $M$ tensor
\be\label{eqn-M}
\left|M\right|^{*2}(s)=\sup\{|M_s(x)|^2: \, x\in \Sigma_s\}
\ee
which depend on $g_t$ and $f(t,x)$ of Hamilton's modified Ricci flow.    Observe that this mass depends only upon the intrinsic metric on $\Sigma$ and not on the mean curvature.
\end{defn}

In Section~\ref{sect-aS} we explore this new notion.
In Lemma~\ref{aSt-inc} we prove that $m_{a\mathbb{S}}(\Sigma,t)$ is nonnegative and increasing in $t$.   In Lemma~\ref{aS-finite}, we prove that the asphericity mass is finite and the limit exists.   
In Proposition~\ref{aS=0} we
 show $m_{a\mathbb{S}}(\Sigma)=0$ if and only if $(\Sigma,g_1)$ is isometric to a standard sphere
$(\mathbb{S}^2, g_{\mathbb{S}^2})$.   

In Section~\ref{sect-Ham} we explore the class of 
asymptotically flat three
dimensional Riemannian manifolds foliated
by Hamilton's modified Ricci flow, denoted $\Ham$.
These manifolds are later used as the
admissible extensions needed to estimate the
Bartnik mass and prove Theorem~\ref{thm-m-B}.
This
class includes the class of asymptotically flat rotationally
symmetric manifolds with nonnegative scalar curvature, $\RotSym$ [Proposition~\ref{prop-rot-sym}].  It
also includes admissible extensions of any $\Sigma$
diffeomorphic to a two sphere with a positive Gauss curvature
and arbitrary positive mean curvature
that have prescribed $0$ scalar curvature [Lemma~\ref{positive-0}]. In addition the class includes admissible extensions
for $\Sigma$ with prescribed scalar curvature, $\bar{R}$, not 
equivalent to 0 as long as $\bar{R}$ satisfies the conditions below
(or the hypothesis of Lemma~\ref{K-R}).   

\begin{defn}\label{defn-Ham}
The class of asymptotically flat three
dimensional Riemannian manifolds foliated
by Hamilton's modified Ricci flow, denoted $M\in \Ham$,
are manifolds $M_{\bar{R}}$ diffeomorphic to 
 $[1,\infty)\times\Sigma$ with metric
 \be
g_{\bar{R}}= u^2 dt^2 + t^2 g
\ee
where $g=g_t$ is defined using the modified Ricci flow
and where $u:[1,\infty)\times\Sigma\to (0,\infty)$
depends uniquely upon $(\Sigma, g_1, H, \bar{R})$.
Here $g_1$ is the metric on $\Sigma$
and  $H:\Sigma \to \mathbb (0,\infty)$ is the mean curvature
of $\Sigma=t^{-1}(1)$:
\be
u(1,x)= 2/H_x 
\ee
and 
$
\bar{R} \in C^\alpha( [1,\infty)\times\Sigma)
$
is a prescribed scalar curvature function
which is asymptotically flat in the sense that
\be \label{R-decays}
\int_1^{\infty} |\bar{R}|^*t^2  dt<\infty \qquad \textrm{ and }
\qquad
\| \bar{R} t^2 \|_{C_{0,\alpha}[t,4t]}\leq \frac{C}{t}  
\ee 
and which has bounded "scalar energy"
with respect to the Ricci flow:
\be\label{R-bound}
C_0(\bar{R})=\sup_{1\leq t<\infty}\left\{ \,\int_{1}^{t}\,\,\left(
\frac{\tau^{2}}{2}\bar{R}-K\right)^{*}\,\,\,\exp\left(\int_{1}^{\tau}\frac{s{|M|^{*}}^{2}}{2}ds \right) \,\,d\tau\,\right\} < H^2/4.
\ee
\textcolor{blue}{Note that $C_0\ge 0$ as seen by taking $t=1$ in the supremum.}
For fixed $(\Sigma, g_1, H)$ that encloses a compact region with nonnegative scalar curvature and positive mean curvature, we denote
\be \label{1.22}
\Ham(\Sigma, g_1,H)=\{M_{\bar{R}}: \,\, \bar{R} \textrm{ satisfies (\ref{R-decays}) and (\ref{R-bound})}\}.
\ee   
and
\be \label{1.23}
\Ham^0(\Sigma, g_1,H)=\{M_{\bar{R}}: \,\, \bar{R}\ge 0 \textrm{ satisfies (\ref{R-decays}) and (\ref{R-bound})}\}.
\ee   
and
\be \label{1.24}
\Ham^0=\bigcup \Ham^0(\Sigma, g_1,H)
\ee
where the union is take over all $(\Sigma, g_1, H)$.
\end{defn}
In Proposition \ref{prop:adm ext}, we prove that $\Ham^0 \subset \mathcal{PM}$.

In \cite{Lin13}, the first author proved that for any
$(\Sigma, g_1, H)$ , $H>0$, and any prescribed $\bar{R}$
satisfying (\ref{R-decays}) and (\ref{R-bound}), one has a unique $M_{\bar{R}} $.
 Thus, for $(\Sigma, g_1, H)$ which is the boundary of a closed 3-dimensional region with nonnegative scalar curvature and positive mean curvature,
\be
\mathrm{m}_{\mathrm{B}}(\Sigma) \le \inf \{\mathrm{m}_{\mathrm{ADM}}(M): \,\, M\in \Ham^0(\Sigma, g_1,H)\}.
\ee

In  Section~\ref{sect-M-R} we prove the following theorem:

\begin{thm}\label{thm-M-R}
If  $M\in \Ham^0$ with
$\Sigma$ a surface of constant positive mean curvature \textcolor{blue}{ satisfying (\ref{R-bound})
and (\ref{1.24}),}
and area $4\pi$ then
\be\label{eqn-M-R}
\mathrm{m}_{\mathrm{ADM}}(M_{\bar{R}}) \leq m_{a\mathbb{S}}(\Sigma)+ \mathrm{m}_{\mathrm{H}}(\Sigma)
+ 
e(M_{\bar{R}},g_{\bar{R}})
\ee
where the additional term
\be\label{eqn-e}
e(M_{\bar{R}},g_{\bar{R}}) 
= \lim_{t\to\infty} e_t(M_{\bar{R}}, g_{\bar{R}})
\ee
where
\be\label{e-t}
e_t(M_{\bar{R}}, g_{\bar{R}})
=
\frac{1}{2}\int_{1}^{t}\frac{ \tau^{2}}{2}\bar{R}^{*}(\tau) E(\tau,t) d\tau
\textrm{ with } E(\tau,t) \textrm{ defined as in (\ref{E-tau}).}
\ee
Here $\left|M\right|^{*2}(s)$, defined as in (\ref{eqn-M}),
depends only on the metric $g_1$ 
and
\be
\bar{R}^{*}(\tau)=\sup\{\bar{R}(\tau, x) : x\in \Sigma\}
\ee
depends only on the prescribed scalar curvature $\bar{R}$,
so that $e(M_{\bar{R}},g_{\bar{R}})$ depends only on $g_1$
and $\bar{R}$. 
\end{thm}

Before proving this theorem,
we first prove that $e_t(M_{\bar{R}}, g_{\bar{R}})$ is nonnegative and increasing in $t$ [Lemma~\ref{et-inc}] and
 the limit in (\ref{eqn-e})
exists and is finite [Lemma~\ref{e-finite}].  

In Section~\ref{sect-0}, we apply this theorem
to prove our main theorem, Theorem~\ref{thm-m-B}.   To do so we prove
$e(M_{\bar{R}}, g_{\bar{R}})=0$ if and only if we prescribe zero 
scalar curvature, $\bar{R}=0$ [Proposition~\ref{e=0}].   
Combining this proposition
with Theorem~\ref{thm-M-R} then implies Theorem~\ref{thm-m-B}.   
\vspace{.2cm}

In Section~\ref{section-mH} we consider rigidity and monotonicity of the Hawking mass of level sets $t=r$ for $M \in \Ham(\Sigma, g_1,H)$.   In \cite{Lin13} the first author proved the Hawking mass is
monotone under the following hypothesis.   Here we combine the
first author's monotonicity result with an analysis of the rigid case:

\begin{thm}\label{thm-min}
Let $(\Sigma, g_1)$ be a surface diffeomorphic to a sphere with 
positive mean curvature (not necessarily constant) and let
$M_{\bar{R}}\subset \Ham^0$, \textcolor{blue}{ satisfying (\ref{R-bound})
and (\ref{1.24}),}
be its admissible
extension
with prescribed scalar curvature $\bar{R}\ge 0$ then
we have monotonicity as in \cite{Lin13}:
\be
\mathrm{m}_{\mathrm{H}}(\Sigma_r) \textrm{ is nondecreasing where }
\Sigma_r= t^{-1}(r).
\ee
Furthermore, if
\be
\mathrm{m}_{\mathrm{ADM}}(M_{\bar{R}})=\mathrm{m}_{\mathrm{H}}(\Sigma)
\ee
then $\bar{R}=0$ everywhere and
$\Sigma$ is isometric to standard sphere, $(\mathbb{S}^2, g_{\mathbb{S}^2})$ and $M_{\bar{R}}$ is rotationally symmetric.
If $\mathrm{m}_{\mathrm{H}}(\Sigma)=0$ then $M_{\bar{R}}$ is isometric to a rotationally symmetric region in Euclidean space.   If $\mathrm{m}_{\mathrm{H}}(\Sigma)=m>0$ then
 $M_{\bar{R}}$ is isometric to a rotationally symmetric region in Schwarzschild
 space of mass $m$.
\end{thm}

Note that we do not assume that $\Sigma$ is a constant mean curvature surface in the hypothesis of this theorem.   This is only
a conclusion in the rigid case.   This theorem was already known
in the rotationally symmetric case to Bartnik \cite[Section 5 ]{Bar93}. 

\vspace{.2cm}
\noindent
{\bf{Acknowledgements:}} The authors would like to thank the Mathematical Sciences Research Institute, and particularly the Women's Program in 
General Relativity in August 2013, for bringing them together in a way which directly lead to this collaboration.  The first author was a speaker in this program funded by MSRI. 
The second author was a visiting research professor at MSRI
during Fall 2013 supported by the NSF under Grant No. 0932078 000.  Additional funding allowing the authors to meet again in January 2014 was provided by the the NSC grant of the first author, NSC-103-2115-M-002-001-MY2,
and the NSF grant of the second author, NSF-DMS DMS-1309360. { \color{blue}{We are very grateful to Hyun-Chul Jang and Pengzi Miao for their careful reading of this paper, for discovering the errors in the publication, and for closely reading our corrections.} }


\section{Hamilton's Ricci Flow and Prescribed Scalar Curvature}\label{sect-review}

In this section, we review the first author's construction of
asymptotically flat manifolds foliated by Hamilton's modified Ricci flow \cite{Lin13}.   Recall that this flow has been defined in
the introduction.   We first recall Hamilton and Chow's theorems
concerning modified Ricci flow first proven in \cite{Ham88}
and \cite{Cho91}.   See also  
Theorems 5.64 and 5.77 from
textbook of Chow and Knopf \cite{CK04}.

\begin{thm} [Hamilton] \cite{Ham88}
Given a surface, $\Sigma$ diffeomorphic to a sphere, with positive Gauss curvature 
there exists a unique solution $g(t)$ to Hamilton's modified Ricci 
flow with $g(1)=g_1$  as defined in (\ref{eq:RF})-(\ref{eq:M}).  The solution $g(t)$ converges exponentially in any $C^k$-norm to a smooth constant-curvature metric $g_{\infty}$ as $t\rightarrow \infty$. 
\end{thm}

The Theorem follows from the exponential decay of $M$. See also \cite[Corollary 5.63]{CK04}. For $k=0,1,2,\cdots$, there are constants $0<c_k, C_k<\infty$ depending only on $g_1$ such that 
\be
|\nabla^k M|\leq C_ke^{-c_k t} 
\ee
which proves that the solution $g(t)$ converges exponentially fast in all $C^k$ to a metric $g_{\infty}$ such that the tensor $M_{\infty}$ vanishes identically. Therefore, we know that the Gauss curvature has decay rate
\be
|K-1|\leq Ce^{-ct} 
\ee
where $c$ and $C$ are constants depending on $g_1$ only. 
\begin{thm} [Chow] \cite{Cho91}
Given a surface, $\Sigma$, with arbitrary Gauss curvature,
there exists a unique solution to Hamilton's modified Ricci 
flow as defined in (\ref{eq:RF})-(\ref{eq:M}).
Furthermore the flow eventually has positive
Gauss curvature so the Gauss curvature and the tensor $M$
both eventually decay exponentially.
\end{thm}

In \cite{Lin13} the first author constructs aymptotically flat 3-metrics of prescribed scalar curvature using parabolic methods.     Given $\left(\Sigma,g_{1}\right)$ a surface of area $4\pi$ which is diffeomorphic to a sphere, an admissible extension is created by taking $M_{\bar{R}}=[1,\infty)\times\Sigma$
equipped with the metric 
\be
g_{\bar{R}}=u^{2}dt^{2}+t^{2}g,
\ee
where $g = g_t$ is the solution of the modified Ricci flow. 
This metric $g_{\bar{R}}=u^{2}dt^{2}+t^{2}g$
has the scalar curvature $\bar{R}$ if and only if $u$ satisfies the
parabolic equation 
\be\label{eq:RFu}
t\frac{\partial u}{\partial t}=\frac{1}{2}u^{2}\Delta u+\frac{t^{2}}{4}\left|M\right|^{2}u+\frac{1}{2}u-\frac{1}{4}\left(2K-t^{2}\bar{R}\right)u^{3},
\ee
where $\Delta$ is the Laplacian with respect to $g$, $K$ is
the Gauss curvature of $g$, $\bar{R}$ is the scalar curvature of $g_{\bar{R}}$, and 
\be
\left| M  \right|^2 = M_{ij}M_{kl}g^{ik}g^{jl}.
\ee  
When the manifold is asymptotically flat with suitable prescribed $\bar{R}$,
the ADM mass is
\be
\mathrm{m}_{\mathrm{ADM}}(M_{\bar{R}}, g_{\bar{R}})=\lim_{t\rightarrow \infty} \mathrm{m}_{\mathrm{H}}(\Sigma_t) =  \lim_{t\rightarrow\infty}\frac{1}{4\pi}\oint_{\Sigma_{t}}\frac{t}{2}(1-u^{-2})d\sigma \label{eq:ADM}
\ee
as  $\Sigma_t$ are nearly round spheres under the Ricci flow.  The fact that $m_{ADM}(M) = \lim_{t\rightarrow \infty} m_H(\Sigma_t)$ where $\Sigma_t$ are nearly round spheres (not just coordinate spheres) was proven by Yuguang Shi, and Guofang Wang, and Jie Wu in \cite{SWW09}. 

Theorems 11-13 of the first author in \cite{Lin13} are combined in the following theorem which provides for the existence and uniqueness of an admissible extension of $\Sigma$ with prescribed scalar curvature $\bar{R}$:

\begin{thm} \cite{Lin13} \label{thm: existence}
Assume that $\bar{R}\in C^{\alpha}(M_{\bar{R}})$ satisfying the decay conditions
\be \label{barR dec}
\int_{1}^{\infty}|\bar{R}|^{*}t^{2}dt<\infty,\mbox{ and } ||\bar{R}t^{2}||_{\alpha,I_{t}}\leq\frac{C}{t} \mbox{ where }  I_{t}=[t,4t], t\geq 1. 
\ee
Let $C_0$ be the nonnegative constant defined by 
\begin{equation} \label{eq:K-1}
C_0=\sup_{1\leq t<\infty}\left\{ -\int_{1}^{t}\left(K-\frac{\tau^{2}}{2}\bar{R}\right)_{*}\exp(\int_{1}^{\tau}\frac{s{|M|^{*}}^{2}}{2}ds)d\tau\right\} <\infty.
\end{equation}
 Then for any function $\phi  \in C^{2,\alpha}(\Sigma)$ satisfying
\be \textcolor{blue}{
0<\phi < \frac{1}{\sqrt{C_0}} \,\,\textrm{ if }\,\, C_0>0 \qquad \textrm{ or } \qquad 0<\phi \,\,\textrm{ if }\,\, C_0=0}
\label{eq:ic-1}
\ee
there is a unique positive solution $u\in C^{2+\alpha}(M_{\bar{R}})$
of (\ref{eq:RFu}) with the initial condition
\begin{equation}
u(1,\cdot)=\phi(\cdot).
\end{equation}
 Moreover, $g_{\bar{R}}$ satisfies the asymptotically flat condition for $t>t_{0}$,
where $t_{0}$ is  some fixed constant with finite ADM mass 
and 
\begin{equation}
\mathrm{m}_{\mathrm{ADM}}(M_{\bar{R}})=\lim_{t\rightarrow\infty}\frac{1}{4\pi}\oint_{\Sigma_{t}}\frac{t}{2}(1-u^{-2})d\sigma.\label{eq:ADM}
\end{equation}
\end{thm}

Here we consider only the special case in which $\Sigma$ is CMC so $\phi$ is a constant. Thus

\begin{thm} \cite{Lin13} \label{simplify}
Assume that $\bar{R}\in C^{\alpha}(M_{\bar{R}})$ satisfying the decay conditions (\ref{barR dec}).
 Then if the mean curvature $H$ of $\Sigma$ satisfies 
 \be \label{eq:ic-11}
H > 2\sqrt{C_0}
\ee
there is a unique positive solution $u\in C^{2+\alpha}(M_{\bar{R}})$
of (\ref{eq:RFu}) with the initial condition
\begin{equation}
u(1,\cdot)=2/H.
\end{equation}
 Then $g_{\bar{R}}$ satisfies the asymptotically flat condition for $t>t_{0}$,
where $t_{0}$ is  some fixed constant with finite ADM mass 
and 
\begin{equation}
\mathrm{m}_{\mathrm{ADM}}(M_{\bar{R}})=\lim_{t\rightarrow\infty}\frac{1}{4\pi}\oint_{\Sigma_{t}}\frac{t}{2}(1-u^{-2})d\sigma.\label{eq:ADM}
\end{equation}
\end{thm}

Theorem~\ref{simplify} immediately implies the existence
of a unique $M_{\bar{R}}$ as described in Definition~\ref{defn-Ham}.

\section{Asphericity Mass}\label{sect-aS}

Here we prove Lemma~\ref{aSt-inc} and Lemma~\ref{aS-finite} which validate the definition of aspherical mass given in Definition~\ref{defn-aS}.   We then prove the key Proposition~\ref{aS=0} which proves the asphericity mass is 0 if and only if the surface is a rescaled standard sphere.  

\begin{lem}\label{aSt-inc} 
If $\Sigma$ is diffeomorphic to a sphere, 
then
 $m_{a\mathbb{S}}(\Sigma,t)$ is nonnegative and increasing in $t$.  
 \end{lem}

\begin{proof}
Since $\Sigma$ is diffeomorphic to a sphere, the Gauss-Bonnet Theorem implies that $\frac{\oint_{\Sigma} K d\sigma}{\oint_{\Sigma}d\sigma}=1$. Thus, $K_* \leq 1$. Together with the fact that 
\be
{\color{blue} 0< }E(\tau,t)=exp\left({-\int_{\tau}^{t}\frac{s\left|M\right|^{*2}(s)}{2}ds}\right) <1, 
\ee 
we see that the integrant of $m_{a\mathbb{S}}(\Sigma,t)$ is nonnegative:
\be
1-K_*exp\left({-\int_{\tau}^{t}\frac{s\left|M\right|^{*2}(s)}{2}ds}\right) 
\geq 1-exp\left({-\int_{\tau}^{t}\frac{s\left|M\right|^{*2}(s)}{2}ds}\right) 
\geq 0.
\ee
Therefore, $m_{a\mathbb{S}}(\Sigma,t)$ is nonnegative and increasing in $t$. 
\end{proof}
 
 \begin{lem}\label{aS-finite}
 The asphericity mass is finite and the limit exists for any $(\Sigma, g_1)$ such that $\Sigma$
is diffeomorphic to a sphere.   
\end{lem}

\begin{proof}
From Lemma \ref{aSt-inc}, we have $m_{a\mathbb{S}}(\Sigma,t)$ is increasing. To show the limit $m_{a\mathbb{S}}(\Sigma)= \lim_{t\to \infty} m_{a\mathbb{S}}(\Sigma,t)$ exists, it suffices to show that $m_{a\mathbb{S}}(\Sigma,t)$ is bounded from above.
First observe that
\begin{eqnarray*}
m_{a\mathbb{S}}(\Sigma,t) 
&=& \frac{1}{2}  \int_{1}^{t}1-K_*(\tau)E(\tau,t) d\tau
{\color{blue} + \frac{1}{2} \int_1^t \frac {\tau |M|^{*2}}{2} E(\tau, t)d\tau}
 \\
&\leq&  \frac{1}{2} \int_{1}^{t} 1- exp\left({-\int_{\tau}^{t}\frac{s\left|M\right|^{*2}(s)}{2}ds}\right) d\tau +  \frac{1}{2}  \int_{1}^{t} \left( 1 - K_*(\tau)\right) 
E(\tau,t)d\tau
{\color{blue} + \frac{1}{2} \int_1^t \frac {\tau |M|^{*2}}{2} E(\tau, t)d\tau}.
\end{eqnarray*}
Using the fact that $E(\tau, t)\le 1$ and
\be \label{e-x-1}
|e^x -1| \leq 2|x| \quad \mbox{for } |x|\leq 1,
\ee
we see that
\begin{eqnarray*}
m_{a\mathbb{S}}(\Sigma,t) 
&\leq& \int_1^t \int_{\tau}^t \frac{s\left|M\right|^{*2}(s)}{2}ds d\tau + \frac{1}{2}  \int_1^t \left( 1 - K_*(\tau) \right) d\tau {\color{blue} + \frac{1}{2} \int_1^t \frac {\tau |M|^{*2}}{2} E(\tau, t)d\tau}\\
&\leq&  \frac{1}{2}  \int_1^t (s-1)s|M|^{*2}(s)ds + \frac{1}{2}  \int_1^t \left( 1 - K_*(\tau) \right) d\tau {\color{blue} + \frac{1}{2} \int_1^t \frac {\tau |M|^{*2}}{2} E(\tau, t)d\tau}\\
&\leq& C. 
\end{eqnarray*}
since $|M|$ and $1-K$ converge to $0$ exponentially under the modified Ricci flow \cite{Cho91, Ham88}. Hence, the lemma follows from the monotonic sequence theorem.  
\end{proof}

\begin{prop}\label{aS=0}
We have $m_{a\mathbb{S}}(\Sigma)=0$ if and only if $(\Sigma,g_1)$ is isometric to a rescaled standard sphere
$(\mathbb{S}^2, g_{\mathbb{S}^2})$. 
\end{prop}

\begin{proof}
Suppose that $(\Sigma,g_1)$ is isometric to $(\mathbb{S}^2, g_{\mathbb{S}^2})$. $|M|\equiv 0$ and $K \equiv 1$ under the Ricci flow. Thus,  $m_{a\mathbb{S}}(\Sigma)=0$. 

Suppose $m_{a\mathbb{S}}(\Sigma)=0$. Since $m_{a\mathbb{S}}(\Sigma,t)$ is nonnegative and increasing in $t$, we have that
\be 
1-K_*(\tau)exp\left({-\int_{\tau}^{t}\frac{s\left|M\right|^{*2}(s)}{2}ds}\right) {\color{blue} + \frac {\tau |M|^{*2}}{2} E(\tau, t) }=0
\ee 
for all $t$ and $\tau$. 
{\color{blue}
Since $$1-K_*(\tau)exp\left({-\int_{\tau}^{t}\frac{s\left|M\right|^{*2}(s)}{2}ds}\right) \geq 0$$ and $$\frac {\tau |M|^{*2}}{2} E(\tau, t) \geq 0,$$ we have 
$$1-K_*(\tau)exp\left({-\int_{\tau}^{t}\frac{s\left|M\right|^{*2}(s)}{2}ds}\right) = 0$$ and $$\frac {\tau |M|^{*2}}{2} E(\tau, t) =0 \textrm{ for all $t$ and $\tau$}.$$ 
}
 It forces that $K = 1$ {\color{blue} and $|M|=0$} for all $t$ and that $(\Sigma,g_1)$ is isometric to a standard sphere by the Uniformization Theorem. 
\end{proof}

\section{The $\Ham$ class of spaces}\label{sect-Ham}

In this section, we study the class of asymptotically flat 
three dimensional Riemannian manifolds foliated by Hamilton's modified Ricci flow defined in Definition~\ref{defn-Ham}.
Recall that the first author has already shown
the existence
of a unique $M_{\bar{R}}$ as described in Definition~\ref{defn-Ham}
(c.f. Theorem~\ref{simplify}).
We now prove this class of spaces contains many
interesting classes of spaces [Lemma~\ref{K-R}, Lemma~\ref{positive-0}] including rotationally symmetric spaces
[Proposition~\ref{prop-rot-sym}].

\begin{prop} \label{prop:adm ext}
Let $(\Sigma, g_1, H)$, $H>0$ be the boundary of a compact manifold of dimension three with nonnegative scalar curvature. For any $M_{\bar{R}}\in \Ham(\Sigma, g_1, H)$,  \textcolor{blue}{ satisfying (\ref{R-bound})
and (\ref{1.22}),} there is no closed minimal surface in $M_{\bar{R}}$. Moreover, $M_{\bar{R}} \in \Ham^0$ is an admissible extension.  
\end{prop}

\begin{proof}
We apply the tangency principle (\cite[Theorem 1.1]{FS01}) by Fonrtenele and Silva. 

Suppose there is a closed minimal surface $S$. There must exist a smallest $t_0$ so that $\Sigma_{t_0}$ is tangent to $S$ at a point $p$. By the assumption $H>0$ and Theorem \ref{thm: existence}, there exists a unique positive solution $u$ and hence mean curvature $H(p)=2/t_0u(p)$ on $\Sigma_{t_0}$ is positive. By the maximum principle (tangency principle),  $S$ and $\Sigma_{t_0}$ coincide in a neighborhood of $p$, which is impossible. 

\end{proof}

\begin{lem}\label{K-R}
Let $(\Sigma, g_1, H)$, $H>0$ be the boundary of a compact manifold of dimension three with nonnegative scalar curvature. Let $\bar{R}$ be
any prescribed scalar curvature satisfying
(\ref{R-decays}) and
\be\label{a}
\bar{R}_{(x,t)} < 2 K_{(x,t)}/t^2
\ee
where $K_{(x,t)}$ is the Gauss curvature of $(\Sigma, g_t)$
obtained by Hamilton's modified Ricci flow, then
we obtain 
\be
M_{\bar{R}} \in \Ham(\Sigma, g_1,H)\in \Ham
\ee
\end{lem}

\begin{proof}
This follows immediately because the assumption in (\ref{a}) which implies the integrand
in the definition of $C_0(\bar{R})$ is nonpositive. The positive mean curvature implies (\ref{R-bound}) which is equivalent to condition (\ref{eq:ic-1}) in Theorem \ref{thm: existence}. By Theorem \ref{thm: existence}, we obtain such a manifold $M_{\bar{R}}$.
\end{proof}

\begin{lem}\label{positive-0}
If
$(\Sigma, g_1)$ with \textcolor{blue}{nonnegative} Gauss curvature
and prescribed $0$ scalar curvature $\bar{R}=0$
then $M_{\bar{R}}$ is defined and $M_{\bar{R}}  \in \Ham^0$. 
\end{lem}

\begin{proof}
Hamilton proved in \cite{Ham88} that
$(\Sigma, g_t)$ has positive Gauss curvature for
all $t\textcolor{blue}{>0}$ and so (\ref{a}) holds for $\bar{R}=0$.
Since $\bar{R}=0$ satisfies (\ref{R-decays}), we
apply Lemma~\ref{K-R} to complete the proof.
\end{proof}

We next prove that the
asymptotically flat rotationally symmetric Riemannian manifolds
of dimension $3$ including the classical rotationally
symmetric gravity wells and black holes lie in $\Ham$:

\begin{defn}  \label{def-rot-sym}
Let $\RotSym$ be the class of complete $3$-dimensional 
asymptotically flat
rotationally symmetric Riemannian manifolds, $(M,g)$,
with
\be
g=(f(r))^2 dr^2 + r^2 g_{S^2}
\ee
of nonnegative scalar curvature $\bar{R}\ge 0$
with no closed interior
minimal hypersurfaces which either have no boundary or have a boundary
which is a stable minimal hypersurface.
\end{defn}

\begin{prop}\label{prop-rot-sym}
If $M\in \RotSym$ is asymptotically flat 
with $r_{min}< 1$ so that
\be\label{mH-decay} 
\exists C>0 \textrm{ such that } |m_H''(r)|< \frac{C}{r^{2}}.
\ee
then
$\bar{R}$ satisfies (\ref{R-decays})
and (\ref{R-bound}).
Thus for any rotationally symmetric
$\Sigma=r^{-1}(t)\in M$ we have 
\be
r^{-1}[t,\infty)\in \Ham.
\ee
\end{prop}

\begin{proof}
In \cite{LeeSormani1}, the second author and Lee proved that there
is a one to one correspondance between manifolds 
$M\in \RotSym$ and nondecreasing continuous functions, 
$m_H: [r_{min}, \infty) \to [0,\infty)$ such that 
\be
m_H(t)<t/2, \quad \lim_{t\to\infty} m_H(t)=\mathrm{m}_{\mathrm{ADM}}(M)<\infty, \quad
m_H(r_{min})=0, \textrm{ and }m_H(t)>0 \textrm{ for } t> r_{min}.
\ee
where $m_H(t)$ denotes the Hawking mass
of the level set $r^{-1}(t)$.

Since
\be
m_H'(t)= t^2 \bar{R}(t)/4.
\ee
we have
\be\label{int-to-m}
\frac{t}{2} > m_H(t)=m_H(1)+\int_{1}^t \frac{t^2 \bar{R}(t)}{4}\, dt.
\ee 
Now $K=1$ and $|M|=0$ in the rotationally symmetric case.  So
\begin{eqnarray}
C_0(\bar{R})&=&\sup_{1\leq t<\infty}\left\{ \int_{1}^{t}\left(
\frac{\tau^{2}}{2}\bar{R}-1\right)d\tau\right\} \\
&=&
\sup_{1\leq t<\infty}\left\{2 m_H(t)-2m_H(1) -(t-1)\right\} \\
&<& t -2(1-H^2/4)-(t-1)=H^2/4.
\end{eqnarray}

Since $\lim_{t\to \infty} m_H(t) <\infty$, by (\ref{int-to-m}) we have
\be
\int_{1}^\infty t^2 \bar{R}(t) \, dt <\infty.
\ee
So we have the first part of (\ref{R-decays}).

Now consider  the weighted H\"{o}lder norm:
\begin{eqnarray}
||\bar{R}(t) t^2||_{\alpha, I_r}
&=&\sup\left\{ {\color{black} t_2^{\alpha}} \frac{|\bar{R}(t_1) t_1^2-\bar{R}(t_2) t_2^2|}{|t_1-t_2|^\alpha}     : t_1\neq t_2 \in [r,4r]\right\}\\
&=&\sup\left\{ {\color{black} t_2^{\alpha}} \frac{|4m_H'(t_1)-4m_H'(t_2)|}{|t_1-t_2|^\alpha}     : t_1\neq t_2 \in [r,4r]\right\}\\
&\le&  {\color{black}16\cdot 3^{1-\alpha} r} \sup_{[r,4r]} |m_H''(t)|
\end{eqnarray}
Assume on the contrary that H\"{o}lder part of
(\ref{R-decays}) is false, then
\be
\lim_{r_j\to\infty} r_j ||\bar{R}(t) t^2||_{\alpha, I_{r_j}}=\infty
\ee
and so
\be
\lim_{r_j\to\infty} r_j^{2} \sup_{[r_j,4r_j]} |m_H''(t)| =\infty
\ee
so
\be
\lim_{r_j\to\infty} r_j^{2} |m_H''(r_j)| =\infty
\ee
which contradicts (\ref{mH-decay}).
\end{proof}

\section{Estimating and Minimizing the ADM mass}
\label{sect-M-R}

Here we prove Theorem~\ref{thm-M-R}.  First
we prove Lemmas~\ref{et-inc} and
\ref{e-finite}.

\begin{lem}\label{et-inc}
Given $\bar{R}\geq 0$, we see that
$e_t(M_{\bar{R}}, g_{\bar{R}})$ is nonnegative and increasing in $t$.
\end{lem}

\begin{proof}

Recall that 
\be
e_t(M_{\bar{R}}, g_{\bar{R}})=
\frac{1}{2}\int_{1}^{t}\frac{\tau^{2}}{2}\bar{R}^{*}(\tau)
E(\tau,t)
d\tau.
\ee
Given $\bar{R} \geq 0$, the integrand of $e_t(M_{\bar{R}}, g_{\bar{R}})$
 is nonnegative. Hence, $e_t(M_{\bar{R}}, g_{\bar{R}})$ is nonnegative and increasing in $t$. 
\end{proof}

\begin{lem}\label{e-finite}
Given $\bar{R} \geq 0$ such that 
\be \label{hyp-e-finite}
\int_1^{\infty}  |\bar{R}|^*t^2 dt < \infty,
\ee
 we see that
the limit $e(M_{\bar{R}},g_{\bar{R}})$ in (\ref{eqn-e})
exists and is finite.
\end{lem}

\begin{proof}
Recall that
\be
e(M_{\bar{R}},g_{\bar{R}}) 
= \lim_{t\to\infty} e_t(M_{\bar{R}}, g_{\bar{R}}).
\ee
With  $\bar{R} \geq 0$, the integrand
of $e_t(M_{\bar{R}}, g_{\bar{R}})$ is nonnegative, so $e_t(M_{\bar{R}}, g_{\bar{R}})$ is increasing in $t$. Moreover, since 
\be
E(\tau,t)=\exp\left(-\int_{\tau}^{t}\frac{s\left|M\right|^{*2}(s)}{2}ds\right) \leq 1,
\ee
 applying (\ref{hyp-e-finite}) we have 
\be
e_t(M_{\bar{R}}, g_{\bar{R}})
\leq \frac{1}{2}\int_{1}^{t}\frac{\tau^{2}}{2}\bar{R}^{*}(\tau)d\tau <\infty.
\ee
Therefore, $e_t(M_{\bar{R}}, g_{\bar{R}})$ is increasing and bounded in $t$, and hence the limit $e(M_{\bar{R}},g_{\bar{R}})$ exists and is finite by the monotonic sequence theorem. 
\end{proof}

We now prove Theorem~\ref{thm-M-R}:

\begin{proof}
By the assumptions, \textcolor{blue}{Lemma~\ref{K-R}} and Theorem \ref{simplify} provides an unique admissible extension $M_{\bar{R}}=[1,\infty) \times \Sigma$  with prescribe scalar curvature $\bar{R}$ is obtained. There exists an unique solution $u \in C^{2+\alpha}(M_{\bar{R}})$ with initial condition $u(1,\cdot)=2/H$ such that the metric 
$$
g_{\bar{R}}=u^2dt^2 + t^2 g_t
$$ satisfies the asymptotically flat condition and finite ADM mass and 
\be
\mathrm{m}_{\mathrm{ADM}} (M_{\bar{R}}) = \lim_{t\rightarrow \infty} \frac{1}{4\pi} \oint_{\Sigma_t} \frac{t}{2}(1-u^{-2})d\sigma.
\ee
Applying the parabolic maximum principle to the parabolic equation of $u^{-2}$ (Lemma 10 in \cite{Lin13})\footnote{This is the source of the error in our original publication.  We were missing the 
second term  in (\ref{Hyun-Chul}).}, we have the following $C^0$ bound:
\be\label{Hyun-Chul}
u^{-2}\left(t\right)\geq \frac{1}{t}\int_{1}^{t}\left(K-\frac{\tau^{2}}{2}\bar{R}\right)_{*}E(\tau,t)d\tau +\frac{1}{t}u^{-2}(1) {\color{blue} E(1,t) }.
\ee
{\color{blue}
Using the following two formulas
\begin{equation*}
\frac 1t \int_1^t 1 d\tau = 1 - \frac 1t
\end{equation*}
and
\begin{eqnarray*}
\frac{ 1}{t} \int_1^t \frac{ \tau |M|^{*2}}{2} E(\tau,t) d\tau 
&=& \frac{ 1}{t} \int_1^t \frac{\partial}{\partial \tau} E(\tau,t) d\tau \\
&=&  \frac{1}{t} \left( 1 - E(1,t) \right),
\end{eqnarray*}
we have
\begin{eqnarray*}
1 = \frac{1}{t} \int_1^t 1 d\tau + \frac {1}{t} \int_1^t \frac{\tau |M|^{*2}}{2} E(\tau,t) d\tau + \frac{1}{t} E(1,t).
\end{eqnarray*}
}

By direction computation,
\begin{eqnarray*}
1-u^{-2} 
 & \leq & 1-\frac{1}{t}\int_{1}^{t}\left(K-\frac{\tau^{2}}{2}\bar{R}\right)_{*}E(\tau,t)d\tau-\frac{H^2}{4t} {\color{blue} E(1,t) }.\\\
 & \leq & 1-\frac{1}{t}\int_{1}^{t}\left(K_*-\frac{\tau^{2}}{2}\bar{R}^{*}\right)E(\tau,t)d\tau-\frac{H^2}{4t} {\color{blue} E(1,t) }.\\
 & = & \frac{1}{t}\int_{1}^{t}1-K_*E(\tau,t) {\color{blue} + \frac {\tau |M|^{*2}}{2} E(\tau,t)}d\tau+\frac{1}{t}\int_{1}^{t}\frac{\tau^{2}}{2}\bar{R}^{*}E(\tau,t)d\tau+\frac{1}{t}\left(1-\frac{H^2}{4}\right){\color{blue} E(1,t) }.
\end{eqnarray*}
Also, the Hawking mass of $\Sigma$ is given by the formula
\be
\mathrm{m}_{\mathrm{H}}(\Sigma) = \sqrt{\frac{A(\Sigma)}{16\pi}}\left(1 - \frac{1}{16\pi}\int_{\Sigma} H^2 d\sigma \right)
= \frac{1}{4\pi} \int_{\Sigma} \frac{1}{2} \left( 1- \frac{H^2}{4} \right) d\sigma.
\ee
Therefore, 
\begin{eqnarray*}
\frac{1}{4\pi} \oint_{\Sigma_t} \frac{t}{2}(1-u^{-2})d\sigma 
&\leq& \frac{1}{2} \int_{1}^{t}1-K_*E(\tau,t) {\color{blue} + \frac {\tau |M|^{*2}}{2} E(\tau,t)} d\tau+\frac{1}{2}\int_{1}^{t}\frac{\tau^{2}}{2}\bar{R}^{*}E(\tau,t)d\tau + \mathrm{m}_{\mathrm{H}}(\Sigma){\color{blue} E(1,t) } \\
&=& m_{a\mathbb{S}}(\Sigma,t) + e_t(M_{\bar{R}}, g_{\bar{R}}) + \mathrm{m}_{\mathrm{H}}(\Sigma){\color{blue} E(1,t) }, 
\end{eqnarray*}
and 
\begin{eqnarray}
\mathrm{m}_{\mathrm{ADM}}(M_{\bar{R}}) 
&\leq& m_{a\mathbb{S}}(\Sigma)  +\mathrm{m}_{\mathrm{H}}(\Sigma){\color{blue} \lim_{t\to \infty}E(1,t)}+ e(M_{\bar{R}}, g_{\bar{R}}) \nonumber\\
&\leq &m_{a\mathbb{S}}(\Sigma)  +\mathrm{m}_{\mathrm{H}}(\Sigma) + e(M_{\bar{R}}, g_{\bar{R}})
\end{eqnarray}
{\color{blue} since $E(1,t) \leq 1$}.
It follows directly by the definition of the Bartnik mass that 
\be
\mathrm{m}_{\mathrm{B}}(\Sigma) \leq m_{a\mathbb{S}}(\Sigma)  + \mathrm{m}_{\mathrm{H}}(\Sigma)+ e(M_{\bar{R}}, g_{\bar{R}}).
\ee
\end{proof}

\section{Proving the Main Theorem}\label{sect-0}
In order to complete the proof of Theorem~\ref{thm-m-B} \textcolor{blue}{we first construct an
extension with prescribed scalar curvature $\bar{R}=0$ using Lemma~\ref{positive-0}, so}
we need only prove the following proposition and
combine it with Theorem~\ref{thm-M-R}:

\begin{prop}\label{e=0}
Given $\bar{R} \geq 0$, we have
$e(M_{\bar{R}},g_{\bar{R}})=0$ if and only if prescribed $\bar{R}=0$.   
\end{prop}

\begin{proof}
Suppose $\bar{R}=0$. It is clear that, by the definition, 
\be
e(M_{\bar{R}},g_{\bar{R}}) = \lim_{t\rightarrow \infty} \frac{1}{2}\int_{1}^{t}\frac{\tau^{2}}{2}\bar{R}^{*}(\tau)e^{-\int_{\tau}^{t}\frac{s\left|M\right|^{*2}(s)}{2}ds}d\tau =0.
\ee
Suppose that $\bar{R} \geq 0$ and $e(M_{\bar{R}},g_{\bar{R}})=0$. Since $e_t(M_{\bar{R}},g_{\bar{R}})$ is nonnegative and increasing in $t$ by Lemma \ref{et-inc}. It follows that $\bar{R}^*(\tau)=0$ for all $\tau$. We therefore conclude that $\bar{R}=0$ since $0=\bar{R}^* \geq \bar{R} \geq 0$. 
\end{proof}

\section{Rigidity and Monotonicity of the Hawking mass}\label{section-mH}

Here we derive a monotonicity formula for Hawking mass which was already known in \cite{Lin13} and then prove Theorem \ref{thm-min}
. 

For  $M_{\bar{R}}=[1,\infty)\times\Sigma$
equipped with the metric 
\be
g_{\bar{R}}=u^{2}dt^{2}+t^{2}g,
\ee
where $g = g_t$ is the solution of the modified Ricci flow. Observe that the mean curvature of a level set of $t$,
$\Sigma_t$, is
\be
H_t(x)= 2/u(t,x).
\ee
Therefore, the Hawking mass (see \cite[Theorem 13]{Lin13}) is given by
\begin{eqnarray}
\mathrm{m}_{\mathrm{H}}(\Sigma_t)
&=& \sqrt{\frac{area\left(\Sigma\right)}{16\pi}}\left(1-\frac{1}{16\pi}\oint_{\Sigma}H^{2}d\sigma \right) \\
&=& \frac{1}{4\pi}\oint_{\Sigma_{t}}\frac{t}{2}(1-u^{-2}(t,x))d\sigma.
\end{eqnarray}

From the Gauss-Bonnet Theorem and (\ref{eq:RFu}), we have the following monotonicity formula provided $\bar{R}\geq 0$. 
\begin{eqnarray}
\frac{d}{dt}\mathrm{m}_{\mathrm{H}}(\Sigma_t) 
&=& \frac{1}{4\pi} \oint_{\Sigma_t} \frac{1}{2}u^{-1}\Delta u + \frac{t^2}{4}|M|^2 u^{-2}+ \frac{t^2}{4}\bar{R} +\left( \frac{1}{2} - \frac{K}{2} \right)d\sigma \\
&=& \frac{1}{8\pi} \oint_{\Sigma_t} \frac{|\nabla u|^2}{u^2} + \frac{t^2}{2} |M|^2u^{-2} + \frac{t^2}{2}\bar{R} d\sigma.\label{mH inc}
\end{eqnarray}

We now prove Theorem~\ref{thm-min}:

\begin{proof}
By the assumptions and Theorem \ref{thm: existence}, an admissible extension $M_{\bar{R}}$ exists and the ADM mass can be obtained by
\be
\mathrm{m}_{\mathrm{ADM}}(M_{\bar{R}}) = \lim_{t \rightarrow \infty} \mathrm{m}_{\mathrm{H}}(\Sigma_t). 
\ee
Given $\bar{R} \geq 0$, $\mathrm{m}_{\mathrm{H}}(\Sigma_t)$ is increasing by the monotonicity formula (\ref{mH inc}). $\mathrm{m}_{\mathrm{ADM}}(M_{\bar{R}}) =\mathrm{m}_{\mathrm{H}}(\Sigma)$ implies that $\frac{d}{dt}\mathrm{m}_{\mathrm{H}}(\Sigma_t)=0$. Hence $\bar{R}=0$, $|M| =0$, and $\nabla u=0$. Since $|M|=0$, 
then $\Sigma$ is isometric to a standard sphere by \cite{Ham88}.  
Since $\nabla u=0$, we have $u(x,t)=u(t)$, so $H$ is
constant and $M_{\bar{R}}$
is rotationally symmetric.   Since $\bar{R}=0$ if $\mathrm{m}_{\mathrm{H}}=m\ge 0$
then $M_{\bar{R}}$ is isometric to a rotationally symmetric
region in $M_{Sch}$ of mass $m$ or Euclidean space  (c.f. Lemma~\ref{app-lem}).   
\end{proof} 

\section{Open Questions}

There are many theorems proven for the rotationally symmetric 
classes of spaces with nonnegative scalar curvature.   It would be
interesting to extend these results to the class of spaces $Ham^0$:

\begin{quest}
What can be said about the 
vacuum solutions of the Einstein equation
which have initial data sets foliated by Hamilton's
modified Ricci flow?
\end{quest}

\begin{quest}
If one fixes $(\Sigma, g_1, H)$, what can be
said about sequences of $M_j \in \Ham^0(\Sigma, g_1, H)$
assuming $\mathrm{m}_{\mathrm{ADM}}(M_j)\le m_0$?
\end{quest}

\section{Appendix on Rotationally Symmetric Spaces}

The following lemma was already basically understood in the
rotationally symmetric setting and is proven here using our
notation for completeness of exposition:

\begin{lem}\label{app-lem}
Given $(\Sigma, g_1)$ isometric to a rescaled standard
sphere and $H>0$ constant and $\bar{R}=0$ and $\mathrm{m}_{\mathrm{H}}(\Sigma)=m\ge 0$ and assume
$\Sigma$ is the boundary of a region $\Omega\subset M$
where $M\subset \mathcal{PM}$
, then
$(M_{\bar{R}}, g_{\bar{R}})$ is a rotationally symmetric
region in a Schwarzschild
space or in Euclidean space with metric:
\be
\bar{g} = \frac{1}{1-2m/t}dt^2 + t^2 g_{\mathbb{S}^2}.
\ee
Since Hawking mass is constant in a Schwarzschild space
we have 
\be\label{m_B-m_H}
\mathrm{m}_{\mathrm{B}}(\Sigma) \le \mathrm{m}_{\mathrm{H}}(\Sigma).
\ee
In particular  $\mathrm{m}_{\mathrm{H}}(\Sigma)\ge 0$,
which implies $H\le 2$.    
\end{lem}

\begin{rmk}
Observe that Shi-Tam have proven in \cite{ShiTam02}
that
\be
\int_{\Sigma} H \, d\sigma \le \int_{\Sigma} H_0 d\sigma
\ee
which in the constant mean curvature case implies
\be
4\pi H \le   \int_{\Sigma} H_0 d\sigma
\ee
and so 
\begin{eqnarray}
\mathrm{m}_{\mathrm{H}}\left(\Sigma\right)&=&\sqrt{\frac{1}{4}}\left(1-\frac{1}{4}H^2 \right)\\
&\ge &\sqrt{\frac{1}{4}}\left(1-\frac{1}{16\pi}
\left(\int_{\Sigma} H_0 d\sigma\right) ^2 \right).
\end{eqnarray}
Furthermore in the rotationally symmetric case $H_0=2$, so
$
4\pi H \le 8\pi
$
and $H\le 2$ just as concluded above.
\end{rmk}

\begin{proof}
Consider
\be 
\bar{g} = u^2(t) dt^2 + t^2g_{\mathbb{S}^2}
\ee
Then 
$\mathrm{m}_{\mathrm{H}}(t)= \frac{t}{2} \left(1 - u^{-2}(t)\right)$ equals the Hawking mass of $\Sigma_t$.     

Observe that $\bar{R}=0$ implies $\mathrm{m}_{\mathrm{H}}'(t)=0$ which
implies $\mathrm{m}_{\mathrm{H}}(t)=m$.   This
\be
m=\frac{t}{2} \left(1 - u^{-2}(t)\right)
\ee
and so
\be
u^2(t)= \frac{1}{1-2m/t}.
\ee
\end{proof}




\end{document}